\documentclass[letterpaper]{article}
\usepackage{hyperref}
\usepackage{amsmath}
\usepackage{amsfonts}
\usepackage{amsthm}
\usepackage{amssymb}
\usepackage{graphicx}
\usepackage{indentfirst}
\usepackage{tikz}
\usepackage[margin = 1.1in]{geometry}

\newcommand{\Z}{\mathbb{Z}}

\newcommand{\C}{\mathbb{C}}

\newcommand{\dX}{-- ++(1, 1.73) -- ++(-1, 1.73)}
\newcommand{\dY}{-- ++(-2, 0) -- ++(-1, -1.73)}
\newcommand{\dZ}{-- ++(1, -1.73) -- ++(2, 0)}
\newcommand{\dx}{-- ++(1, -1.73) -- ++(-1,  -1.73)}
\newcommand{\dy}{-- ++(1, 1.73) -- ++(2, 0)}
\newcommand{\dz}{-- ++(-2, 0) -- ++(-1, 1.73)}

\newcommand{\vbone}[1]{
	\begin{scope}
		\draw[clip] {#1} \dX\dX\dX\dY\dx\dx\dZ -- cycle;
		\draw[line width = 3pt, color = blue, pattern={Lines[line width = 0.5pt, angle=90,distance=1.5]},pattern color=blue] {#1} \dX\dX\dX\dY\dx\dx\dZ -- cycle;
		\draw[line width = 3pt, color = blue, fill = blue, fill opacity = 0.33] {#1} \dX\dX\dX\dY\dx\dx\dZ -- cycle;
	\end{scope}
}

\newcommand{\snake}[2]{
	\begin{scope}
		\draw[clip] {#1} {#2} -- cycle;
		\draw[line width = 3pt, color = purple, pattern={
			Hatch[line width = 0.5pt, angle=45, distance=2.5, xshift=.1pt]}, pattern color=purple] {#1} {#2} -- cycle;
		\draw[line width = 3pt, color = purple, fill = purple, fill opacity = 0.33] {#1} {#2} -- cycle;
	\end{scope}
}
\newcommand{\stone}[2]{
	\begin{scope}
		\draw[clip] {#1} {#2} -- cycle;
		\draw[line width = 3pt, color = green, pattern={
			Dots[radius=0.66, angle=45, distance=2.5, xshift=.1pt]}, pattern color=green] {#1} {#2} -- cycle;
		\draw[line width = 3pt, color = green, fill = green, fill opacity = 0.33] {#1} {#2} -- cycle;
	\end{scope}
}

\newcommand{\negtile}[2]{
	\begin{scope}
		\draw[clip] {#1} {#2} -- cycle;
		\draw[line width = 4pt, color = lime, pattern = {Lines[line width = 0.7mm, distance=2mm, angle=0]}, pattern color = lime] {#1} {#2} -- cycle;
	\end{scope}
}

\newcommand{\lr}[1]{\left\langle{#1}\right\rangle}
\newcommand{\tile}[1]{\downharpoonleft\hspace{-1.2ex}[{#1}]\hspace{-1.2ex}\upharpoonright}

\usetikzlibrary{decorations.pathreplacing}
\usetikzlibrary{decorations.footprints}
\usetikzlibrary{patterns.meta}

\title{Stones, Bones, and Snakes: Tilability of the hexagonal grid via the double dimer model}

\usepackage{todonotes}
\author{Leigh Foster\thanks{University of Waterloo, \href{mailto:mathematicleigh@gmail.com}{mathematicleigh@gmail.com}.}}

\definecolor{yellow}{RGB}{253, 231, 37}
\definecolor{lime}{RGB}{94, 201, 98}
\definecolor{green}{RGB}{33, 145, 140}
\definecolor{blue}{RGB}{59, 82, 139}
\definecolor{purple}{RGB}{68, 10, 84}

\newtheorem{theorem}{Theorem}
\newtheorem{conjecture}[theorem]{Conjecture}
\newtheorem{corollary}[theorem]{Corollary}

\newtheorem{lemma}[theorem]{Lemma}
\newtheorem{remark}[theorem]{Remark}

\theoremstyle{definition}

\newtheorem{example}[theorem]{Example}
\newtheorem{definition}[theorem]{Definition}

\DeclareMathOperator{\SL}{SL}

\usepackage[backend=bibtex]{biblatex}
\begin{document}
	
	\maketitle
	
	\begin{abstract}
		The question of whether a given region can be successfully filled by a finite set of tiles has been commonly studied, and there are many available arguments for whether a given finite region can be tiled. We can show that there is no domino tiling of the mutilated chessboard via a coloring argument, and a slightly more subtle argument for other two-colored square-grid regions using a height function of Thurston. 
		In this paper, we examine finite regions of the hexagonal grid and a set of tiles known as the \textit{stone}, \textit{bone}, and \textit{snake}. Using matrices in $\text{SL}_2(\C)$, we exhibit a new necessary criterion for a region to have a signed tiling by these tiles. This originally arose in a study of the double dimer model.  
		%Now, instead of using a coloring argument \cite{coloring argument paper}, or a height function \cite{height function paper}, or similar established necessary conditions for tilability, we introduce a new method involving matrices in $\text{SL}_2(\C)$ that arose from a study of the double dimer model, based on the work of \cite{young:squish} and \cite{kenyon:conformal}.  Knowledge of the double dimer model is unnecessary for this paper, though a good introduction to the topic is \cite{kenyon:intro}. 
	\end{abstract}

\section{Introduction}

In 1990, Conway and Lagarias \cite{conway-lagarias} introduced conditions for the tilability of a finite region in a regular hexagonal grid using combinatorial group theory. Based on this work, Thurston \cite{thurston} introduced new necessary conditions on the tilability of such a region using height functions. Both papers used a given set of tiles (later called bones and stones), and in this paper we introduce one additional tile, called the snake. We give a new set of necessary conditions on the tilability of a finite region of the regular hexagonal grid, using technology from the single and double dimer models (as introduced in \cite{kenyon:conformal}, \cite{young:squish}, and \cite{foster-young}). 

A portion of these results have previously appeared in the PhD dissertation \cite{foster:thesis}.  We would like to thank Benjamin Young and Hanna Mularczyk for helpful conversations. 

\subsection{Tilability}

First, we introduce the concept of tilability.  Throughout the paper, we will use $R$ to mean a region enclosed by a loop in the hexagonal grid. $R$ can also be thought of as a simply-connected collection of faces of the grid. For a given graph edge $e$, we use $\partial R(e)$ to mean the path around $R$ when traveling in a counterclockwise direction starting at $e$. See the left of Figure \ref{crescenttiling} for an example of $R$ and $\partial R$. 

%introduced a combinatorial group-theoretic way of studying 
%Tilings of (subregions of) the plane have been frequently studied in mathematics. In \cite{conway-lagarias} they tiled some stuff, and then \cite{thurston} studied the height functions of some of these regions. Similar tiling work has been done by \cite{propp:something},  and has been continued in his paper \cite{propp:problems}. In this work, we introduce a third kind of tile called the snake. 

%In \cite{foster-young}, we reintroduce the squish map, originally studied in \cite{young:squish}, and prove that the squish map is a measure-preserving transformation between instances of the single and double dimer models on honeycomb graphs. Here we continue that work, proving results about tilability of a region on the hexagonal grid. This is also like \cite{conway-lagarias} stuff. Fun times. 

\begin{definition}  A \emph{signed tiling} of a region $R$ is a collection of tiles, each with a weight of $+1$ or $-1$, covering $R$ in such a way that the total contribution at each hexagon inside $R$ is 1, and the total contribution of each hexagon outside $R$ is zero. 
	\label{signed tiling}
\end{definition}

\begin{definition}
	A \emph{(standard) tiling} of a region $R$ is a signed tiling where every tile has weight 1. So a standard tiling has no overlapping tiles, and no tiles going outside of $R$. 
\end{definition}

\begin{definition}
	\label{def:tiles}
	A \emph{bone} is the union of three collinear adjacent hexagons in $H$. (These were originally named \textit{tribones} in \cite{conway-lagarias}, and were given the name \textit{bones} in \cite{jesse-propp}.)  A \emph{stone} is the union of three hexagons in $H$ which all share a common vertex. (These we called $T_2$ in \cite{thurston}, and were given the name \textit{stone} in \cite{jesse-propp}.)  A \emph{snake} is the union of four hexagons in an ``S'' shape (named in \cite{foster-young}). See Figure \ref{fig:stone-bone-snake}.
\end{definition}

\begin{figure}[h!]
	\centering
	\begin{tikzpicture}[scale = 0.233, rotate = 0]
		\clip (-16.5, -4) rectangle (17.5, 13);
\draw[dotted] (-16.5, -4) rectangle (17.5, 13);
\draw (-15, -3*1.73) \dX\dX\dX\dX\dX\dX\dX\dX\dX\dX\dX;
\draw (-12, -4*1.73) \dX\dX\dX\dX\dX\dX\dX\dX\dX\dX\dX;
\draw (-9, -3*1.73) \dX\dX\dX\dX\dX\dX\dX\dX\dX\dX\dX;
\draw (-6, -4*1.73) \dX\dX\dX\dX\dX\dX\dX\dX\dX\dX\dX;
\draw (-3, -3*1.73) \dX\dX\dX\dX\dX\dX\dX\dX\dX\dX\dX;
\draw (0, -4*1.73) \dX\dX\dX\dX\dX\dX\dX\dX\dX\dX\dX;
\draw (3, -3*1.73) \dX\dX\dX\dX\dX\dX\dX\dX\dX\dX\dX;
\draw (6, -4*1.73) \dX\dX\dX\dX\dX\dX\dX\dX\dX\dX\dX;
\draw (9, -3*1.73) \dX\dX\dX\dX\dX\dX\dX\dX\dX\dX\dX;
\draw (12, -4*1.73) \dX\dX\dX\dX\dX\dX\dX\dX\dX\dX\dX;
\draw (15, -3*1.73) \dX\dX\dX\dX\dX\dX\dX\dX\dX\dX\dX;

\draw (18, 0*1.73) \dY\dY\dY\dY\dY\dY\dY\dY\dY\dY\dY;
\draw (18, 2*1.73) \dY\dY\dY\dY\dY\dY\dY\dY\dY\dY\dY;
\draw (18, 4*1.73) \dY\dY\dY\dY\dY\dY\dY\dY\dY\dY\dY;
\draw (18, 6*1.73) \dY\dY\dY\dY\dY\dY\dY\dY\dY\dY\dY;
\draw (18, 8*1.73) \dY\dY\dY\dY\dY\dY\dY\dY\dY\dY\dY\dY;
\draw (18, 10*1.73) \dY\dY\dY\dY\dY\dY\dY\dY\dY\dY\dY\dY;
\draw (18, 12*1.73) \dY\dY\dY\dY\dY\dY\dY\dY\dY\dY\dY\dY;
\draw (6, 12*1.73) \dY\dY\dY\dY\dY\dY\dY\dY\dY\dY\dY;
\draw (12, 12*1.73) \dY\dY\dY\dY\dY\dY\dY\dY\dY\dY\dY;
\draw (12, 16*1.73) \dY\dY\dY\dY\dY\dY\dY\dY\dY\dY\dY;

\stone{(-9, 1.73)}{\dX\dX\dY\dY\dZ\dZ}
\snake{(15, 1*1.73)}{ \dX\dz\dY\dz\dY\dZ\dZ\dy\dZ}
\vbone{(0,0)}
	\end{tikzpicture}
	\caption[Stone, bone, and snake]{One possible orientation of the stone (left), bone (center), and snake (right) tiles}
	\label{fig:stone-bone-snake}
\end{figure}
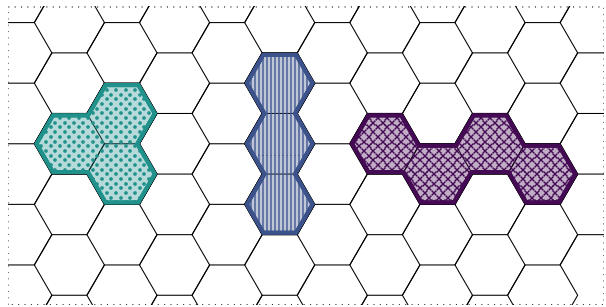

\begin{example}
	
	See Figure \ref{crescenttiling} for an example of a signed tiling using exactly one bone, one stone, and one snake. We start by placing the teal (dotted) stone at the top of the image, then the purple (diamond grid) snake at the bottom. Note that each of these tiles leaves the desired blue dashed region so it cannot be a regular tiling. Finally, place a lime bone (lined) with negative weight, resulting in a weight of 1 inside the dashed region, and a weight of 0 everywhere else. 
	\begin{figure}[h!]
	\centering
	\resizebox{!}{2in}{
		\begin{tikzpicture}[scale = 0.3, rotate = 120]
			\draw[clip] (4, 4) circle (8);
\draw[line width = 3pt, color = blue, fill = blue!10!white] (9, 1.73) \dX\dX\dz\dY\dY\dZ\dy\dx\dZ;
\draw (-6, -4*1.73) \dX\dX\dX\dX\dX\dX\dX\dX;
\draw (-3, -3*1.73) \dX\dX\dX\dX\dX\dX\dX\dX;
\draw (0, -4*1.73) \dX\dX\dX\dX\dX\dX\dX\dX;
\draw (3, -3*1.73) \dX\dX\dX\dX\dX\dX\dX\dX;
\draw (6, -4*1.73) \dX\dX\dX\dX\dX\dX\dX\dX;
\draw (9, -3*1.73) \dX\dX\dX\dX\dX\dX\dX\dX;
\draw (12, 0*1.73) \dY\dY\dY\dY\dY\dY\dY\dY;
\draw (12, 2*1.73) \dY\dY\dY\dY\dY\dY\dY\dY;
\draw (12, 4*1.73) \dY\dY\dY\dY\dY\dY\dY\dY;
\draw (12, 6*1.73) \dY\dY\dY\dY\dY\dY\dY\dY;
\draw (12, 8*1.73) \dY\dY\dY\dY\dY\dY\dY\dY;
\draw (12, 10*1.73) \dY\dY\dY\dY\dY\dY\dY\dY;
\draw[line width = 3pt, dotted, color = blue] (9, 1.73) \dX\dX\dz\dY\dY\dZ\dy\dx\dZ;
		\end{tikzpicture}
		\begin{tikzpicture}[scale = 0.3, rotate = 120]
			\draw[clip] (4, 4) circle (8);
\draw (-6, -4*1.73) \dX\dX\dX\dX\dX\dX\dX\dX;
\draw (-3, -3*1.73) \dX\dX\dX\dX\dX\dX\dX\dX;
\draw (0, -4*1.73) \dX\dX\dX\dX\dX\dX\dX\dX;
\draw (3, -3*1.73) \dX\dX\dX\dX\dX\dX\dX\dX;
\draw (6, -4*1.73) \dX\dX\dX\dX\dX\dX\dX\dX;
\draw (9, -3*1.73) \dX\dX\dX\dX\dX\dX\dX\dX;
\draw (12, 0*1.73) \dY\dY\dY\dY\dY\dY\dY\dY;
\draw (12, 2*1.73) \dY\dY\dY\dY\dY\dY\dY\dY;
\draw (12, 4*1.73) \dY\dY\dY\dY\dY\dY\dY\dY;
\draw (12, 6*1.73) \dY\dY\dY\dY\dY\dY\dY\dY;
\draw (12, 8*1.73) \dY\dY\dY\dY\dY\dY\dY\dY;
\draw (12, 10*1.73) \dY\dY\dY\dY\dY\dY\dY\dY;
\stone{(9, 1.73)}{\dX\dX\dY\dY\dZ\dZ}
\snake{(0, 0)}{ \dy\dX\dy\dX\dY\dY\dx\dY\dZ}
\negtile{(0,0)}{\dy\dy\dX\dY\dY\dY\dZ}
\draw[line width = 3pt, color = blue] (9, 1.73) \dX\dX\dz\dY\dY\dZ\dy\dx\dZ;
	\end{tikzpicture}}
	\caption[Crescent-moon tiling]{A signed tiling of a crescent-moon region. The purple (diamond grid) snake and teal (dotted) stone are positive, and the lined green bone is negative.}
	\label{crescenttiling}
\end{figure}
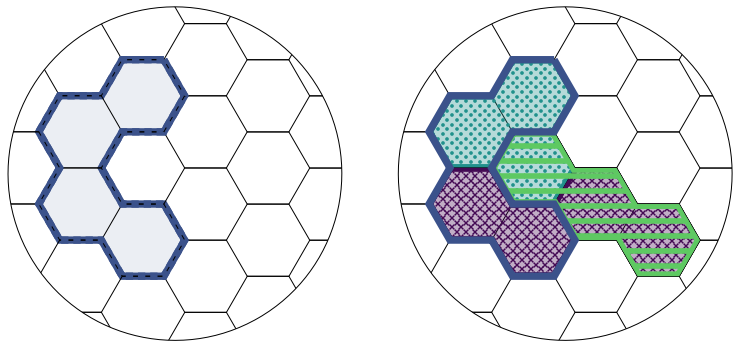
\end{example}

Finally, to give an explicit description of $\partial R(e)$, we assign a label to each graph edge. We place $\alpha$ on all northeast oriented edges, $\beta$ on all northwest oriented edges, and $\gamma$ on all horizontal (westward) oriented edges, as in Figure \ref{fig:matrix-weights}. This is similar to the combinatorial group theoretic description given in \cite{conway-lagarias}. We give a full description of $\partial R$ for each tile in Figure \ref{alltiles}. 

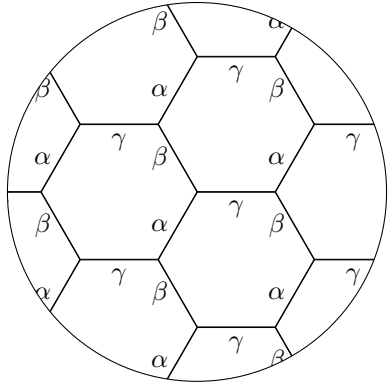
\begin{figure}
	\centering
	\begin{tikzpicture}[scale = 1.8]
		\def\hexagon{+(0:0.575cm) -- +(60:0.575cm) -- +(120:0.575cm) -- +(180:0.575cm) -- +(240:0.575cm) -- +(300:0.575cm) -- cycle}
		\def\hexagonweights{+(0:0.575cm) -- node [anchor = east]{$\beta$} +(60:0.575cm) -- node [anchor = north]{} +(120:0.575cm) -- node [anchor = west]{} +(180:0.575cm) -- node [anchor = west]{} +(240:0.575cm) -- node [anchor = north]{$\gamma$} +(300:0.575cm) -- node [anchor = east]{$\alpha$}cycle;}
		\def \squish {+(0: 1cm) -- 
			++(60: 1cm) -- +(60: .155cm) -- +(240: 0cm) -- 
			++(180:1cm) -- +(120: .155cm) -- +(240: 0cm) -- 
			++(240: 1cm) -- +(180: .155cm) -- + (240: 0cm) -- 
			++(300: 1cm) -- +(240: .155cm) -- +(240: 0cm) -- 
			++(360: 1cm) -- +(300: .155cm) -- +(240: 0cm) -- 
			++(420: 1cm) -- +(360: .155cm) -- +(240: 0cm)}
		
		\clip[draw] (0.575,1) circle (1.4cm);

		\foreach \x in {-2,...,2}
		\foreach \y\ in {-2,...,2}
		\draw (1.73*\x,\y) \hexagonweights;

		\foreach \x in {-2,...,2}
		\foreach \y\ in {-2,...,2}
		\draw (1.73*\x + 1.73*0.5,\y + 0.5) \hexagonweights;
	\end{tikzpicture}
	\caption{The assignment of $\alpha$, $\beta$, and $\gamma$ to every edge of the hexagonal grid}
	\label{fig:matrix-weights}
\end{figure}

\begin{figure}[h!]\footnotesize
	$$\begin{array}{c c c}
		\begin{tikzpicture}[scale = 0.15]
			\def \hex {+(0, 0) -- +(2, 0) -- +(3, 1*1.73) -- +(2, 2*1.73) -- +(0, 2*1.73) -- +(-1, 1*1.73) -- +(0, 0)};
			\draw[very thin] (0, 0)     \hex;
			\draw[very thin] (-3, 1*1.73) \hex;
			\draw[very thin] (-6, 2*1.73) \hex;
			\draw[color = blue, line width = 2pt] (0, 0) -- ++(2, 0) -- ++(1, 1.73) -- ++(-1, 1.73) -- ++(-2, 0) -- ++(-1, 1.73) -- ++(-2, 0) -- ++(-1, 1.73) -- ++(-2, 0) -- ++(-1, -1.73) -- ++(1, -1.73) -- ++(2, 0) -- ++(1, -1.73) -- ++(2, 0) -- ++(1, -1.73) -- cycle;
			\node at (2, 0) {$\bullet$};
		\end{tikzpicture}
		&
		\begin{tikzpicture}[scale = 0.15]
			\def \hex {+(0, 0) -- +(2, 0) -- +(3, 1*1.73) -- +(2, 2*1.73) -- +(0, 2*1.73) -- +(-1, 1*1.73) -- +(0, 0)};
			\draw[very thin] (0, 0)     \hex;
			\draw[very thin] (0, 2*1.73) \hex;
			\draw[very thin] (0, 4*1.73) \hex;
			\draw[color = blue, line width = 2pt] (0, 0) -- ++(2, 0) -- ++(1, 1.73) -- ++(-1, 1.73) -- ++(1, 1.73) -- ++(-1, 1.73) -- ++(1, 1.73) -- ++(-1, 1.73) -- ++(-2, 0) -- ++(-1, -1.73) -- ++(1, -1.73) -- ++(-1, -1.73) -- ++(1, -1.73) -- ++(-1, -1.73) -- cycle;
			\node at (2, 0) {$\bullet$};
		\end{tikzpicture}
		&	
		\begin{tikzpicture}[scale = 0.15]
			\def \hex {+(0, 0) -- +(2, 0) -- +(3, 1*1.73) -- +(2, 2*1.73) -- +(0, 2*1.73) -- +(-1, 1*1.73) -- +(0, 0)};
			\draw[very thin] (-2, 0)     \hex;
			\draw[very thin] (1, 1*1.73) \hex;
			\draw[very thin] (4, 2*1.73) \hex;
			\draw[color = blue, line width = 2pt] (0, 0) -- ++(-2, 0) -- ++(-1, 1.73) -- ++(1, 1.73) -- ++(2, 0) -- ++(1, 1.73) -- ++(2, 0) -- ++(1, 1.73) -- ++(2, 0) -- ++(1, -1.73) -- ++(-1, -1.73) -- ++(-2, 0) -- ++(-1, -1.73) -- ++(-2, 0) -- ++(-1, -1.73) -- cycle;
			\node at (0, 0) {$\bullet$};
		\end{tikzpicture}
		\\
		(\gamma^{-1}\beta^{-1})^3\alpha^{-1}(\gamma\beta)^3 \alpha 
		&
		\gamma^{-1}(\beta^{-1}\alpha^{-1})^3\gamma(\beta\alpha)^3 
		& 
		\beta^{-1}(\alpha^{-1}\gamma)^3 \beta (\alpha\gamma^{-1})^3
	\end{array}$$
	$$\begin{array}{c c}
		\begin{tikzpicture}[scale = 0.15]
			\def \hex {+(0, 0) -- +(2, 0) -- +(3, 1*1.73) -- +(2, 2*1.73) -- +(0, 2*1.73) -- +(-1, 1*1.73) -- +(0, 0)};
			\draw[very thin] (0, 0)     \hex;
			\draw[very thin] (3, -1*1.73) \hex;
			\draw[very thin] (3, 1*1.73) \hex;
			\draw[color = blue, line width = 2pt] (0, 0) -- ++(2, 0) -- ++(1, -1.73) -- ++(2, 0) -- ++(1, 1.73) -- ++(-1, 1.73) -- ++(1, 1.73) -- ++(-1, 1.73) -- ++(-2, 0) -- ++(-1, -1.73) -- ++(-2, 0) -- ++(-1, -1.73) -- ++(1, -1.73) -- cycle;
			\node at (5, -1.73) {$\bullet$};
		\end{tikzpicture}
		&
		\begin{tikzpicture}[scale = 0.15]
			\def \hex {+(0, 0) -- +(2, 0) -- +(3, 1*1.73) -- +(2, 2*1.73) -- +(0, 2*1.73) -- +(-1, 1*1.73) -- +(0, 0)};
			\draw[very thin] (-2, 0)     \hex;
			\draw[very thin] (-5, -1*1.73) \hex;
			\draw[very thin] (-5, 1*1.73) \hex;
			\draw[color = blue, line width = 2pt] (0, 0) -- ++(-2, 0) -- ++(-1, -1.73) -- ++(-2, 0) -- ++(-1, 1.73) -- ++(1, 1.73) -- ++(-1, 1.73) -- ++(1, 1.73) -- ++(2, 0) -- ++(1, -1.73) -- ++(2, 0) -- ++(1, -1.73) -- ++(-1, -1.73) -- cycle;
			\node at (-3, -1.73) {$\bullet$};
		\end{tikzpicture}
		\\
		(\gamma^{-1}\beta^{-1})^2(\alpha^{-1}\gamma)^2(\beta\alpha)^2 
		&
		(\beta^{-1}\alpha^{-1})^2(\gamma\beta)^2(\alpha\gamma^{-1})^2
	\end{array}$$
	$$\begin{array}{c c c}
		\begin{tikzpicture}[scale = 0.15]
			\def \hex {+(0, 0) -- +(2, 0) -- +(3, 1*1.73) -- +(2, 2*1.73) -- +(0, 2*1.73) -- +(-1, 1*1.73) -- +(0, 0)};
			\draw[very thin] (0, -2*1.73) \hex;
			\draw[very thin] (3, -3*1.73) \hex;
			\draw[very thin] (6, -2*1.73) \hex;
			\draw[very thin] (9, -3*1.73) \hex;
			\draw[color = blue, line width = 2pt] (0, 0) -- ++(2, 0) -- ++(1, -1.73) -- ++(2, 0) -- ++(1, 1.73) -- ++(2, 0) -- ++(1, -1.73) -- ++(2, 0) -- ++(1, -1.73)  -- ++(-1, -1.73) -- ++(-2, 0) -- ++(-1, 1.73) -- ++(-2, 0) -- ++(-1, -1.73) -- ++(-2, 0) -- ++(-1, 1.73) -- ++(-2, 0) -- ++(-1, 1.73) -- cycle;
			\node at (11, -3*1.73) {$\bullet$};
		\end{tikzpicture}
		&
		\begin{tikzpicture}[scale = 0.15]
			\def \hex {+(0, 0) -- +(2, 0) -- +(3, 1*1.73) -- +(2, 2*1.73) -- +(0, 2*1.73) -- +(-1, 1*1.73) -- +(0, 0)};
			\draw[very thin] (-2, 0) \hex;
			\draw[very thin] (-2, 2*1.73) \hex;
			\draw[very thin] (-5, 3*1.73) \hex;
			\draw[very thin] (-5, 5*1.73) \hex;
			\draw[color = blue, line width = 2pt] (0, 0) -- ++(-2, 0) -- ++(-1, 1.73) -- ++(1, 1.73) -- ++(-1, 1.73) -- ++(-2, 0) -- ++(-1, 1.73) -- ++(1, 1.73) -- ++ (-1, 1.73) -- ++(1, 1.73) -- ++(2, 0) -- ++(1, -1.73) -- ++(-1, -1.73) -- ++(1, -1.73) -- ++(2, 0) -- ++(1, -1.73) -- ++(-1, -1.73) -- ++(1, -1.73) -- cycle;
			\node at (0, 0) {$\bullet$};
		\end{tikzpicture}
		&
		\begin{tikzpicture}[scale = 0.15]
			\def \hex {+(0, 0) -- +(2, 0) -- +(3, 1*1.73) -- +(2, 2*1.73) -- +(0, 2*1.73) -- +(-1, 1*1.73) -- +(0, 0)};
			\draw[very thin] (0, 0) \hex;
			\draw[very thin] (3, 1*1.73) \hex;
			\draw[very thin] (6, 0*1.73) \hex;
			\draw[very thin] (9, 1*1.73) \hex;
			\draw[color = blue, line width = 2pt] (0, 0) -- ++(2, 0) -- ++(1, 1.73) -- ++(2, 0) -- ++(1, -1.73) -- ++(2, 0) -- ++(1, 1.73) -- ++(2, 0) -- ++(1, 1.73)  -- ++(-1, 1.73) -- ++(-2, 0) -- ++(-1, -1.73) -- ++(-2, 0) -- ++(-1, 1.73) -- ++(-2, 0) -- ++(-1, -1.73) -- ++(-2, 0) -- ++(-1, -1.73) -- cycle;
			\node at (11, 1.73) {$\bullet$};
		\end{tikzpicture}
		\\
		(\alpha\gamma^{-1}\beta^{-1}\gamma^{-1})^2\beta^{-1}(\alpha^{-1}\gamma\beta\gamma)^2\beta\ \  
		&
		(\gamma^{-1}\beta^{-1}\alpha^{-1}\beta^{-1})^2\alpha^{-1}(\gamma\beta\alpha\beta)^2\alpha\ 
		&
		(\gamma^{-1}\alpha\gamma^{-1}\beta^{-1})^2\alpha^{-1}(\gamma\alpha^{-1}\gamma\beta)^2\alpha 
		\\
		\begin{tikzpicture}[scale = 0.15]
			\def \hex {+(0, 0) -- +(2, 0) -- +(3, 1*1.73) -- +(2, 2*1.73) -- +(0, 2*1.73) -- +(-1, 1*1.73) -- +(0, 0)};
			\draw[very thin] (0, 0) \hex;
			\draw[very thin] (-3, 1*1.73) \hex;
			\draw[very thin] (-3, 3*1.73) \hex;
			\draw[very thin] (-6, 4*1.73) \hex;
			\draw[color = blue, line width = 2pt] (0, 0) -- ++(2, 0) -- ++(1, 1.73) -- ++(-1, 1.73) -- ++(-2, 0) -- ++(-1, 1.73) -- ++(1, 1.73) -- ++(-1, 1.73) -- ++(-2, 0) -- ++(-1, 1.73) -- ++(-2, 0) -- ++(-1, -1.73) -- ++(1, -1.73) -- ++(2, 0) -- ++(1, -1.73) -- ++(-1, -1.73) -- ++(1, -1.73) -- ++(2, 0) -- cycle;
			\node at (2, 0) {$\bullet$};
		\end{tikzpicture}
		&
		\begin{tikzpicture}[scale = 0.15]
			\def \hex {+(0, 0) -- +(2, 0) -- +(3, 1*1.73) -- +(2, 2*1.73) -- +(0, 2*1.73) -- +(-1, 1*1.73) -- +(0, 0)};
		\draw[very thin] (0, 0) \hex;
		\draw[very thin] (0, 2*1.73) \hex;
		\draw[very thin] (3, 3*1.73) \hex;
		\draw[very thin] (3, 5*1.73) \hex;
		\draw[color = blue, line width = 2pt] (0, 0) -- ++(2, 0) -- ++(1, 1.73) -- ++(-1, 1.73) -- ++(1, 1.73) -- ++(2, 0) -- ++(1, 1.73) -- ++(-1, 1.73) -- ++ (1, 1.73) -- ++(-1, 1.73) -- ++(-2, 0) -- ++(-1, -1.73) -- ++(1, -1.73) -- ++(-1, -1.73) -- ++(-2, 0) -- ++(-1, -1.73) -- ++(1, -1.73) -- ++(-1, -1.73) -- cycle;
		\node at (2, 0) {$\bullet$};
		\end{tikzpicture}
		&
		\begin{tikzpicture}[scale = 0.15]
			\def \hex {+(0, 0) -- +(2, 0) -- +(3, 1*1.73) -- +(2, 2*1.73) -- +(0, 2*1.73) -- +(-1, 1*1.73) -- +(0, 0)};
			\draw[very thin] (-2, 0) \hex;
			\draw[very thin] (1, 1*1.73) \hex;
			\draw[very thin] (1, 3*1.73) \hex;
			\draw[very thin] (4, 4*1.73) \hex;
			\draw[color = blue, line width = 2pt] (0, 0) -- ++(-2, 0) -- ++(-1, 1.73) -- ++(1, 1.73) -- ++(2, 0) -- ++(1, 1.73) -- ++(-1, 1.73) -- ++(1, 1.73) -- ++(2, 0) -- ++(1, 1.73) -- ++(2, 0) -- ++(1, -1.73) -- ++(-1, -1.73) -- ++(-2, 0) -- ++(-1, -1.73) -- ++(1, -1.73) -- ++(-1, -1.73) -- ++(-2, 0) -- cycle;
			\node at (3, 1.73) {$\bullet$};
		\end{tikzpicture}
		\\
		\gamma^{-1}(\beta^{-1}\gamma^{-1}\beta^{-1}\alpha^{-1})^2\gamma(\beta\gamma\beta\alpha)^2 
		&
		\beta^{-1}(\alpha^{-1}\beta^{-1}\alpha^{-1}\gamma)^2\beta(\alpha\beta\alpha\gamma^{-1})^2 
		& 
		\gamma^{-1}(\beta^{-1}\alpha^{-1}\gamma\alpha^{-1})^2\gamma(\beta\alpha\gamma^{-1}\alpha)^2
		
	\end{array}$$\normalsize
	\caption[Stone, bone, and snake tiles]{All possible orientations of the stone, bone, and snake tiles, as well as one representative product describing $[\partial R]$. The top row is all three orientations of bones, the second row both orientations of stones, and the bottom two rows all six orientations of the snake tile.}
	\label{alltiles}
\end{figure}
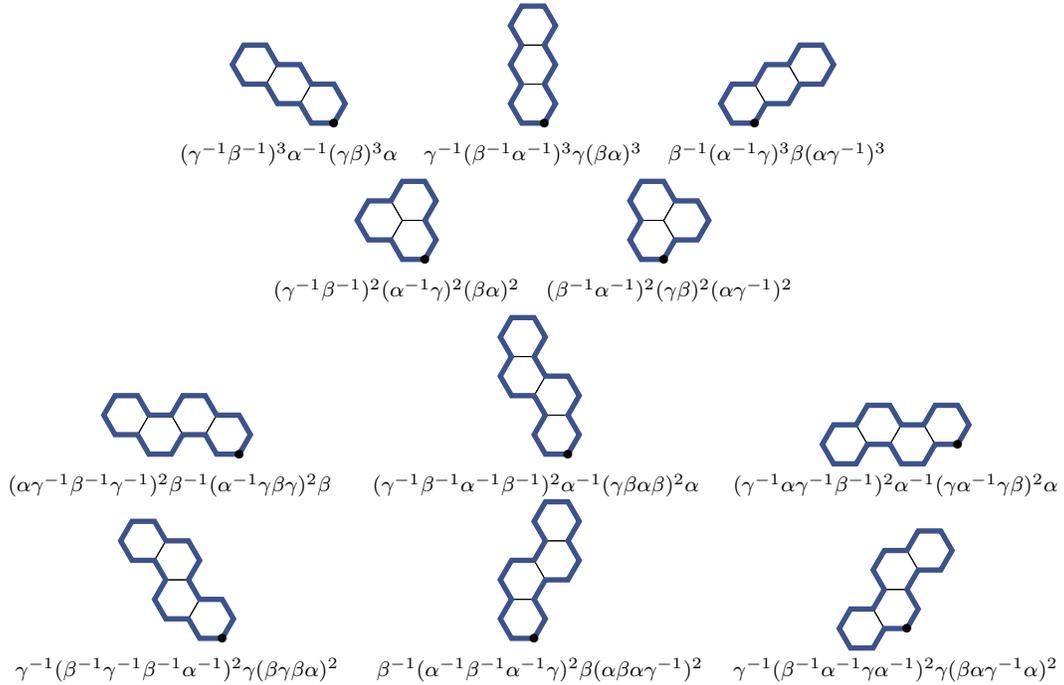

Although we can always specify a path beginning at a particular edge, it is often more convenient to talk about a path in more generality. Doing so will lead us to be able to talk about tilings from a group-theoretic point of view. The following three definitions are due to Conway and Lagarias, in \cite{conway-lagarias}. 

 \begin{definition}
	The notation $[\partial{R}]$ is defined to be to be the \emph{combinatorial boundary} - that is, all of the conjugacy classes in the free group that contains all oriented boundaries $\partial R(e)$ of edges in the tile, i.e. $$[\partial R] = \{W\partial R(e)W^{-1} \vert W \in F\}$$ where $F$ is the free group on the generators given by the edges in the grid. 
	\label{conwaydef}
\end{definition}

\begin{definition}
	Let $F$ be the free group on generators given by the set of edge labels of the graph. Then for a set of tiles $\Sigma$, the \textit{tile group} $T(\sigma)$ is the smallest normal subgroup of $F$ containing the combinatorial boundaries of all tiles in $\Sigma$. 
	
	$$T(\Sigma) = N(\lr{\partial R_i(e_i) \mid 1 \leq i \leq m}) = \lr{W\partial R_i(e_i)W^{-1} \mid W \in F, 1 \leq i \leq m}$$	
\end{definition}

\begin{definition}
	Let $C$, called the cycle group, be defined as the commutator subgroup of $F$ (where $F$ is defined as above). Or, $C = [F:F]$. Then we define the \textit{tile homotopy group} to be $$h(\Sigma) = C/T(\Sigma).$$
	\label{homotopy}
\end{definition}

Note that $C$ consists of all (allowable) paths around closed loops in the grid,  and $T(\Sigma)$ represents the paths that can be deformed
 to the empty path by picking up or laying down tiles. 

\subsection{The $\SL_2{\C}$ connection}

Our main results also use technology from the single and double dimer models, so we aim to give a brief overview of the necessary material. We will first discuss the dimer model in some generality, and then see how it applies to the hexagonal grid and tilings. 

\begin{definition}
	A \textit{single dimer cover}, or \textit{perfect matching}, is a subset of graph edges such that every vertex has degree one. 
\end{definition}

\begin{definition}
	The \emph{weight} of a graph is an assignment $\nu: E \to \mathbb{R}_{\geq 0}$ of real numbers onto each edge, $e$, of the graph, where $E$ is the set of edges in a graph $G$. The weight of a dimer configuration $D$ is $$w(D) = \prod_{e \in D} w(e).$$ 
\end{definition}

\begin{definition}
	A \textit{double dimer configuration} is a subset of graph edges such that every vertex has degree two. A double dimer configuration consists entirely of closed loops and doubled edges. 
\end{definition}

\begin{definition}
	This definition is due to \cite{kenyon:conformal}. Let $G = (V,E)$ be a bipartite graph with a \emph{scalar weight} $w:E \to \mathbb C$ as well as an \emph{$\text{SL}_2$ connection}: a map $\Gamma:E \rightarrow SL_2(\mathbb C)$. Then the contribution of a double dimer configuration $DD$ is defined to be 
	\[
	\left(\prod_{e \in DD} w(e) \right)
	\times\!\!\!\!\!\! \prod_{\text{closed loops $L \in DD$}}
	\text{Tr}\left(\prod_{e \in L} \Gamma(e)\right).\]
\end{definition}

One thing that is important to note is that the first product in the above definition is over all edges of the double dimer configuration, including doubled edges. This means that doubled edges are in fact counted twice. The second product, over closed loops, does not include the doubled edges, only nontrivial cycles. Finally, note that $Tr(w) = Tr(w^{-1})$, so the formula still holds even if we traverse a loop backwards. We will generally refer to the scalar portion as the weight, and the trace of the matrix product to be the contribution. 

\begin{figure}[h!]
	\centering
	\resizebox{5cm}{!}{
	\begin{tikzpicture}[scale = 0.15]
		 

\def \hex {+(0, 0) -- +(2, 0) -- +(3, 1*1.73) -- +(2, 2*1.73) -- +(0, 2*1.73) -- +(-1, 1*1.73) -- +(0, 0)};
\def\up#1{
	\begin{scope}[shift={#1}]
		\draw [line width = 2.5pt, lime, line cap=round] (1, -1*1.73) -- (2, 0*1.73);
	\end{scope}
}
\def\down#1{
	\begin{scope}[shift={#1}]
		\draw [line width = 2.5pt, lime, line cap=round] (-1, -1.73) -- (-2, 0);
	\end{scope}
}
\def\flat#1{
	\begin{scope}[shift={#1}]
		\draw [line width = 2.5pt, lime, line cap=round] (-1, 1.73) -- (1, 1.73);
	\end{scope}
}

\flat{(-3,1.73)};
\flat{(3, 5*1.73)};
\flat{(0, 4*1.73)};
\flat{(3, -5*1.73)};
\flat{(6, 4*1.73)};
\flat{(0, -2*1.73)};
\flat{(3, -1*1.73)};
\flat{(6, -2*1.73)};
\flat{(9, -1*1.73)};

\down{(-3,1*1.73)};
\down{(-3, -1*1.73)};
\down{(0, 4*1.73)};
\down{(0, -2*1.73)};
\down{(6, 4*1.73)};
\down{(6, 2*1.73)};
\down{(6, -2*1.73)};
\down{(12, 4*1.73)}; 
\down{(12, 2*1.73)}; 

\up{(-3, 1*1.73)};
\up{(0, -2*1.73)};
\up{(6, 4*1.73)};
\up{(6, 2*1.73)};
\up{(0, 2*1.73)};
\up{(0, 4*1.73)};
\up{(6, -2*1.73)};
\up{(9, -1*1.73)};
\up{(-6, 4*1.73)};

\draw[very thin, black] (2, 4*1.73) \hex;
\draw[very thin, black] (2, 2*1.73) \hex;
\draw[very thin, black] (2, 0*1.73) \hex;
\draw[very thin, black] (2,-2*1.73) \hex;
\draw[very thin, black] (2,-4*1.73) \hex;
\draw[very thin, black] (-4, 0*1.73) \hex;
\draw[very thin, black] (-4,-2*1.73) \hex;
\draw[very thin, black] (-4, 2*1.73) \hex;
\draw[very thin, black] (-1, -3*1.73) \hex;
\draw[very thin, black] (-1, 3*1.73) \hex;
\draw[very thin, black] (-1, 1*1.73) \hex;
\draw[very thin, black] (-1,-1*1.73) \hex;
\draw[very thin, black] (5, -3*1.73) \hex;
\draw[very thin, black] (5, 3*1.73) \hex;
\draw[very thin, black] (5, 1*1.73) \hex;
\draw[very thin, black] (5,-1*1.73) \hex;
\draw[very thin, black] (8,-2*1.73) \hex;
\draw[very thin, black] (8, 0*1.73) \hex;
\draw[very thin, black] (8, 2*1.73) \hex;
	\end{tikzpicture} }
\qquad
\resizebox{5cm}{!}{
		\begin{tikzpicture}[scale = 0.15]
			\input{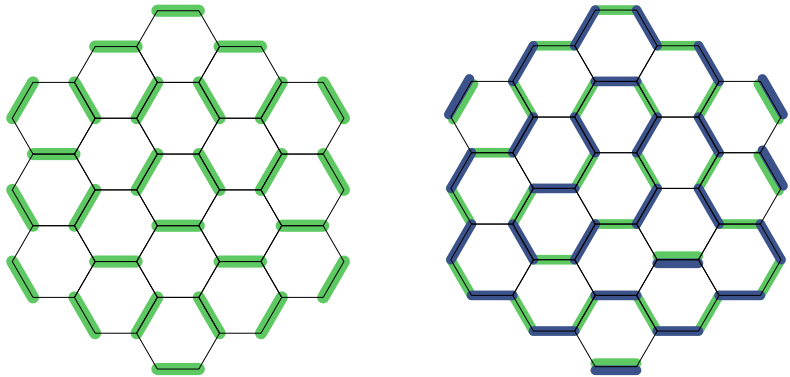}
		\end{tikzpicture}
	}
	\caption{A $3 \times 3 \times 3$ hexagonal grid showing a single dimer cover at left, and a double dimer configuration on the right}
	\label{single-double-dimers}
\end{figure}

We can now define $\alpha$, $\beta$, and $\gamma$, as used above, to be the following matrices in $\SL_2(\C)$. 

\begin{definition}
	Define the matrices $\alpha$, $\beta$, and $\gamma$ as follows, with $\omega$ a primitive 3rd root of unity. 
	
	$$\displaystyle\alpha := 
\begin{bmatrix}
	\omega^7 & 0 \\
	0 & \omega^5
\end{bmatrix}
\qquad 
\beta := \displaystyle
\begin{bmatrix} 
	\omega^7 & \omega^3 \\
	0 & \omega^5
\end{bmatrix}
\qquad
\gamma := \displaystyle
\begin{bmatrix} 
	\omega^5 & 0 \\
	\omega^3 & \omega^7
\end{bmatrix}.
$$
\end{definition}

A more general version of these matrices was used in \cite{foster-young} to prove a partition function for a particular $\Z_2 \times \Z_2$ partition function.

Since each tile is a finite region of the hexagonal grid, we can view $\partial R$ as its corresponding double dimer cover. Then each tile has associated with it a matrix product. Thus we can apply the language of the dimer model to study tilings.

\section{Tilability results}

To determine whether a simply-connected region $R$ is tilable, we first discuss the times where we may remove a tile from $R$. Our first requirement is that doing so may not disconnect $R$ (we want to avoid breaking the region up). So, for example, the bone shown in Figure \ref{badbone} would not be a proper choice for a first tile to remove. Second, we require that the tile being removed is on the boundary of $R$ so that we do not turn a simply-connected region into a non-simply-connected region (we want to avoid punctures). 

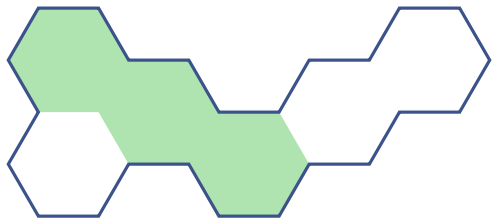
\begin{figure}[h!]
	\centering
	\begin{tikzpicture}[scale = 0.4]
		\draw[opacity = 0, fill opacity = 50, fill = lime!50] (0, 0) \dX\dz\dz\dY\dZ\dZ\dZ -- cycle;
\draw[very thick, color = blue] (0,0) \dy\dy\dX\dY\dY\dz\dz\dY\dx\dZ\dy\dZ -- cycle;
	\end{tikzpicture}
	\caption{A bad choice of bone to remove}
	\label{badbone}
\end{figure} 

\begin{remark}
	The contribution of any tile in Figures \ref{alltiles} and \ref{alltilesXYZ} is $I$, the $2 \times 2$ identity matrix. This is a direct matrix computation, but is an important fact used in the following results. 
\end{remark}
\begin{lemma}
	Removing a bone or a snake from a region $R$ does not change the contribution of $\partial R$, and removing a stone negates the product. 	\label{outandback}
\end{lemma}
\begin{figure}[h!]
	\centering
	\begin{tikzpicture}[scale = 0.5]
		\clip (-1, 3) rectangle (-19, -8);

\draw[very thick, fill = gray, fill opacity = 0.2, color = lime] (0,0) \dx\dz\dz\dY\dY\dY\dZ\dy\dy -- ++(1, 1.73) -- ++(-1, -1.73) \dY\dY\dz\dz\dz -- (-100, 0) -- (-100, -100) -- (0, -100) -- (10, 0) -- cycle;
\draw[very thick, color = lime, fill = yellow] (-6, -2*1.73) \dX\dY\dY\dY\dZ\dy\dy -- cycle;
\draw[very thick, color = yellow, dashed] (-5, -1*1.73) -- ++(-1, -1.73) \dY\dY\dz;
\node[color = green] at (-5.1, -1.9) {$\bullet$};
\fill[color = blue, decorate, decoration = {footprints, foot of = felis silvestris, stride length = 25pt, foot length = 5pt, foot sep = -4pt}] (0,0.2) \dx\dz\dz\dY\dY\dY\dZ -- ++(1, 1.73) -- ++(1.9, 0) \dy -- ++(1, 1.73) -- ++(0.2, -0.4) -- ++(-1, -1.73) \dY\dY -- ++(-0.2, 0) \dz\dz\dz;
	\end{tikzpicture}
	\caption[Out-and-back path along a bone]{The path along $\partial R$, then all the way around a bone, then back out along the bone to $\partial R$ again}
	\label{fig:outandback}
\end{figure}
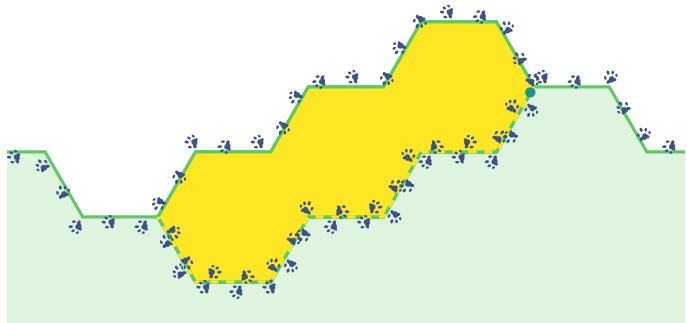
\begin{proof}
	Consider any region $R$, and the contribution of $\partial R$ before removing a tile. This path will be a matrix product, and since we are allowed cyclic permutations we can choose the one starting at a chosen boundary point of $R$ not in or adjacent to the tile we want to remove. Let $X_1$ be the product of connection matrices from that chosen point until the start of the tile, $Y_1$ the common boundary of the tile and $R$, and then $X_2$ the rest of the boundary of $R$ up until the chosen point. So we can write our matrix product as $X_2Y_1X_1$. Now let the path around the tile in its entirety be $y_1 \cdot \ldots \cdot y_j \cdot \ldots \cdot y_k$, where $Y_1 = y_1 \cdot \ldots \cdot y_j$ and $Y_2 = y_{j+1} \cdot \ldots \cdot y_k$ (and the edge corresponding to $y_k$ shares a vertex with the edge corresponding to $y_1$). So, in other words, $Y_1$ is the common boundary between the tile and $R$, and $Y_2$ is the boundary of the tile that is in the interior of $R$. 
	
	Now we can consider walking around $R$ via $X_1$, then $Y_1$, and then instead of immediately following $X_2$ we follow the path corresponding to $Y_2$ and its inverse. So we insert $I = Y_2^{-1}Y_2$, and get $X_2Y_2^{-1}Y_2Y_1X_1$.

	The path around each tile is given in Figure \ref{alltiles}, and if we compute out those matrix products, then the bones and snakes have a product equal to $I$, the $2 \times 2$ identity matrix, and the stones $-I$. So removing a bone or a snake allows us to replace $Y_1Y_2$ with $I$, and a stone allows us to replace $Y_1Y_2$ with $-I$. So if the tile we would like to remove is a bone or snake, we have $X_2Y^{-1}_2Y_2Y_1X_1 = X_2Y^{-1}_2X_1$, and for the stone $X_2Y^{-1}_2Y_2Y_1X_1 = X_2Y_2^{-1}(-I)X_1 = -X_2Y_2^{-1}X_1$. 
\end{proof}

\begin{theorem}
	If a region $R$ has a signed tiling by the stone, bone, and snake tiles, then $\partial R$ contributes $\pm I$, the $2 \times 2$ identity matrix. Hence, $\partial R = \pm I$ is a necessary condition for $R$ to be signed tilable by the stone, bone, and snake tiles. 
	\label{trace2}
\end{theorem}

\begin{proof}
	We prove this by induction. Start with the base case of a single tile, which is tilable by definition. Then assume that some region $R$ is tilable by the stone, bone, and snake tiles, and that $\partial R = I$. Lemma \ref{outandback} says that we may remove a tile from the boundary of $R$ and that performing this action affects $\partial R$ only by a possible sign alternation. So if $R'$ is the region with the tile removed, then $\partial R = \partial R'$ if the tile removed was a stone or snake, and $\partial R = -\partial R'$ if the tile removed was a stone. Since $R$ is tilable, we can continue removing one tile at a time in this manner until we have a single tile remaining. 
\end{proof}

\begin{corollary}
	If a region $R$ is (regularly) tilable by stones, bones, and snakes, then $\partial R$ contributes $\pm I$. 
	\label{trace2-tilable}
\end{corollary}

\begin{proof}
	Because signed tilability is a necessary criterion for a region to have a regular tiling, then Corollary \ref{trace2-tilable} follows directly from Theorem \ref{trace2}. 
\end{proof}
\begin{corollary}
	If a signed tiling for $R$ exists using only bones and snakes, then $\partial R = I$. 
	\label{tilingtotrace}
\end{corollary}

\begin{proof}
	Trace through the proof of Theorem \ref{trace2} forbidding the use of a stone tile. At every step in the argument $\partial R = \partial R' = \partial R'' = \ldots = I$ since removing a bone or a snake does not change the matrix product. 
\end{proof}

\begin{corollary}
	If $R$ is signed tilable with bones and snakes and requires exactly one stone, then $\partial R = -I$. 
\end{corollary}

\begin{proof}
	If a signed tiling for $R$ exists that requires the use of exactly one stone, then the step that requires removing the stone flips the sign of $\partial R$. 
\end{proof}

\begin{conjecture}
	The sign of $\partial R$ corresponds with the parity of the minimum number of stones needed for $R$ to be signed tilable when adding tiles of either weight along the boundary. So $\partial R = (-1)^{\text{\#\{stones needed\}}}\cdot I$.
	\label{full tiling}
\end{conjecture}

\begin{example}
	\label{ex2x2x2}
	One interesting example is the signed tiling of the $2 \times 2 \times 2$ hexagonal region, seen in Figure \ref{2x2x2}. This is one of the first examples we found that gave us some trouble, so it is worth diving into here in some detail. One useful way to talk about signed tiling is to use ``adding'' a tile to mean adding a tile with positive weight, and ``removing'' a tile to mean adding a tile with negative weight.
	
	First, we need to be careful about the way in which we tile these regions for our loop monodromy idea to hold. See the first two pictures on Figure \ref{2x2x2}. We can begin with the four tiles in the leftmost picture (three bones and one snake). By Corollary \ref{tilingtotrace}, then $\partial R'' = I$. We then remove the top stone ($\partial R' = -I$), and then the bottom stone, so we have that $\partial R'' = I$, as desired. We can make a similar argument for the right-most two images, removing only bones and snakes (not stones), so this removal does not change the sign of $\partial R$. The important part of the left and right images was that we could add or remove each tile individually along the boundary at each step - we still had a boundary that was a single closed loop around a simply-connected interior. However, the signed tiling of the middle picture behaves differently. 
	
	 Although the middle image exhibits a signed tiling (in the sense of \cite{conway-lagarias}) using one stone, it cannot be made by only adding or removing tiles along a boundary edge. If we start by adding the blue snake and purple bone tiles, we have the region shown in the left of Figure \ref{onestonetile}. Then, if we remove the green stone, we are left with the shape on the right of Figure \ref{onestonetile} We then cannot add in the bone we need without crossing the horizontal line between the two partial-bones. 
\end{example}

\begin{figure}[h!]
	\centering
	\resizebox{!}{1.1in}{
	$\begin{array}{c|c|c}
		\begin{tikzpicture}[scale = 0.25]
		\vbone{(0,0)}
		\vbone{(0, 6*1.73)}
		\vbone{(-3, 1.73)}
		\snake{(-6, 4*1.73)}{\dX\dy\dX\dX\dY\dx\dY\dx\dZ}
		\draw[dotted, line width = 4pt, yellow] (0, 4*1.73) \dX\dX\dz\dY\dY\dx\dZ\dZ\dy;
		\end{tikzpicture}
		\begin{tikzpicture}[scale = 0.25]
		\stone{(0,0)}{\dX\dX\dY\dY\dZ\dZ}
		\stone{(0,8*1.73)}{\dX\dX\dY\dY\dZ\dZ}
		\draw (0, 4*1.73) \dX\dX\dz\dY\dY\dx\dZ\dZ\dy;
		\draw[dotted, line width = 4pt, yellow] (0, 4*1.73) \dX\dX\dz\dY\dY\dx\dZ\dZ\dy;
		\end{tikzpicture}
		&
		\begin{tikzpicture}[scale = 0.25]
		\vbone{(0, 0)}
		\vbone{(-3, -1*1.73)}
		\snake{(-6, 0)} {\dX\dy\dX\dX\dY\dx\dY\dx\dZ};
		\draw[dotted, line width = 4pt, yellow] (0, 0*1.73) \dX\dX\dz\dY\dY\dx\dZ\dZ\dy;
		\node at (0, -3.8*1.73) {\textcolor{white}{.}};
		\end{tikzpicture}
		\begin{tikzpicture}[scale = 0.25]
		\stone {(0, 4*1.73)} {\dX\dz\dY\dx\dZ\dy};
		\draw[dotted, line width = 4pt, yellow] (0, 0*1.73) \dX\dX\dz\dY\dY\dx\dZ\dZ\dy;
		\node at (0, -3.8*1.73) {\textcolor{white}{.}};
		\end{tikzpicture}
		&
		\begin{tikzpicture}[scale = 0.25]
			\input{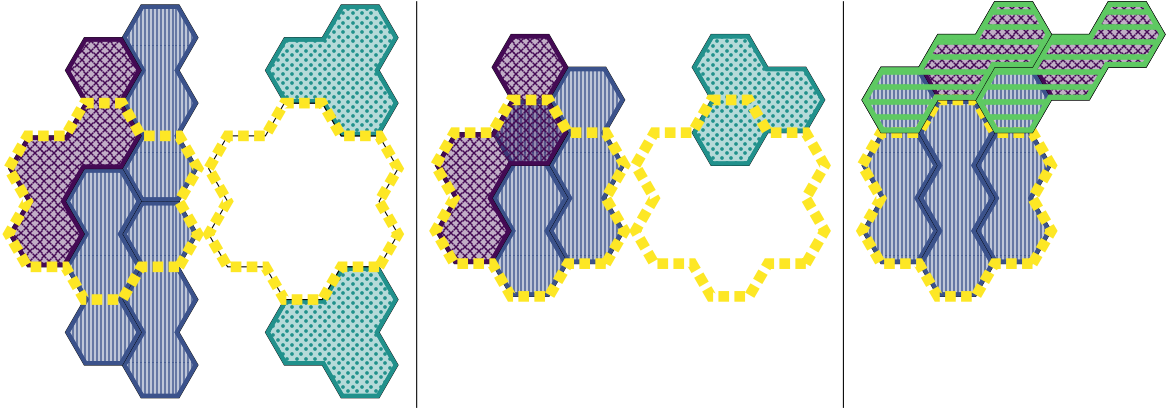}
			\node at (0, -3.8*1.73) {\textcolor{white}{.}};
			\end{tikzpicture}
	\end{array}$}
	\caption[Various $2\times2\times2$ signed tilings]{Three different ways of tiling the $2 \times 2 \times 2$ hexagonal grid. The first has two stones, then one stone, but the final example shows that a signed tiling exists using no stones. Thus $\partial R = I$.}
	\label{2x2x2}
\end{figure}
\begin{figure}[h!]
\centering
\resizebox{!}{2in}{
\begin{tikzpicture}[scale = 0.25]
	\input{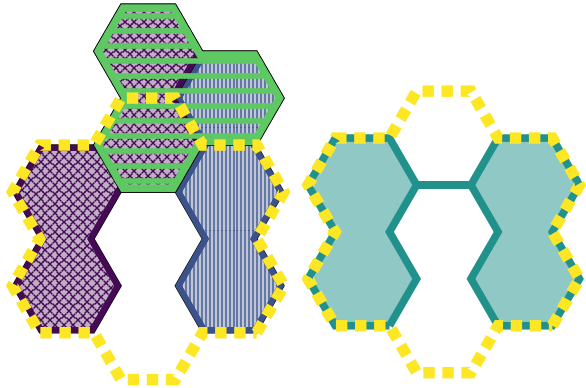}
\end{tikzpicture}
\begin{tikzpicture}[scale = 0.25]
	\draw[line width = 2pt, color = green, fill = green, fill opacity = 0.5] (0, 0) \dX\dy\dx\dZ\dX\dX\dY\dz\dY\dx\dZ;
\node at (0, -1.73) {};
\draw[line width = 3pt, dotted, color = yellow] (6, 0) \dX\dX\dz\dY\dY\dx\dZ\dZ\dy;
\end{tikzpicture}
}
\caption[Why you have to add/remove tiles on a boundary]{The middle tiling from Figure \ref{2x2x2}. }
\label{onestonetile}
\end{figure}

\section{Tiling groups}

From the previous section, we used matrices $\alpha$, $\beta$ and $\gamma$, which are edge connection matrices $\SL_2(\C)$. These can be used as generators to form a subgroup of $\SL_2(\C)$. However, some edge sequences are not possible without leaving the hexagonal grid (for example, we cannot walk $\alpha \to \gamma = \gamma\alpha$, as seen in Figure \ref{fig:matrix-weights}). %To further the connection between our $SL_2(\C)$ matrices and boundary paths, we exploit the bipartite coloring of the hexagonal grid. 
Note that the set of all possible moves on the hexagonal grid consists of the following moves, with edge labels given by $\alpha$, $\beta$, and $\gamma$. \begin{align*}
	\alpha 	&\mapsto \beta 		&\qquad \alpha &\mapsto \gamma^{-1} &\qquad \alpha^{-1} &\mapsto \beta^{-1}&\qquad \alpha^{-1} &\mapsto \gamma\\
	\beta 	&\mapsto \gamma &\qquad \beta &\mapsto \alpha &\qquad 
	\beta^{-1} &\mapsto \gamma^{-1} &\qquad \beta^{-1} &\mapsto \alpha^{-1}\\
	\gamma &\mapsto \alpha^{-1} &\qquad \gamma &\mapsto \beta & \qquad
	\gamma^{-1} &\mapsto \alpha &\qquad \gamma^{-1} &\mapsto \beta^{-1}
\end{align*}

We now choose those paths of length two that begin at a shaded vertex (and as such also end at a shaded vertex). These relations are given in Definition \ref{def:XYZ}, and shown in Figure \ref{fig:XYZ}. This re-defining of generators allows us to have a more direct correspondence between a particular subgroup of $\SL_2(\C)$ and words corresponding to closed boundary paths on the hexagonal grid. Fortunately, this also allows us to simplify notation. An updated labeling of our tile set can be seen in Figure \ref{alltilesXYZ}. 

\begin{definition}
	\label{def:XYZ}
	Let $X := \beta\alpha$, $Y := (\alpha)^{-1}\gamma$, and $Z := (\beta)^{-1}(\gamma)^{-1}$. Additionally, let $x:= X^{-1}$, $y:= Y^{-1}$, and $z:= Z^{-1}$.  Additionally, let $G$ be the subgroup of $\SL_2(\C)$ generated by $X$, $Y$ and $Z$. 
	
	For a visual representation of these relations on the hexagonal grid, see Figure \ref{fig:XYZ}. 
\end{definition}

\begin{figure}[h!]
	\centering
	\begin{tikzpicture}[scale = 0.5]
		\draw[line width = 4pt, green, ->] (5, 1.73) \dX;
		\draw[line width = 4pt, green, ->] (11, 1*1.73) \dY;
		\draw[line width = 4pt, green, ->] (8, 4*1.73) \dZ;
		\draw[line width = 4pt, lime, ->] (14, 4*1.73) \dx;
		\draw[line width = 4pt, lime, ->] (17, 1*1.73) \dy;
		\draw[line width = 4pt, lime, ->] (20, 4*1.73) \dz;
		\node at (5, 2*1.73) {$X$};
		\node at (10, 0.5*1.73) {$Y$};
		\node at (9.5, 3.5*1.73) {$Z$};
		\node at (14, 3*1.73) {$x$};
		\node at (18.5, 1.2*1.73) {$y$};
		\node at (18.5, 4.5*1.73) {$z$};
		\foreach \i in {0, 6, 12, 18}{
			\foreach \j in {0, 2*1.73, 4*1.73} {
				\draw[fill = black] (\i + 2, \j) circle (0.15); 
				\draw[fill = white] (\i + 6, \j) circle (0.15); 
			}
		}
		\foreach \i in {3, 9, 15, 21}{
			\foreach \j in {2*1.73/2, 6*1.73/2, 10*1.73/2} {
				\draw[fill = white] (\i, \j) circle (0.15); 
				\draw[fill = black] (\i + 2, \j) circle (0.15); 
			}
		}
	\end{tikzpicture}
	
	\caption{The moves $X$, $Y$, and $Z$}
	\label{fig:XYZ}
\end{figure}
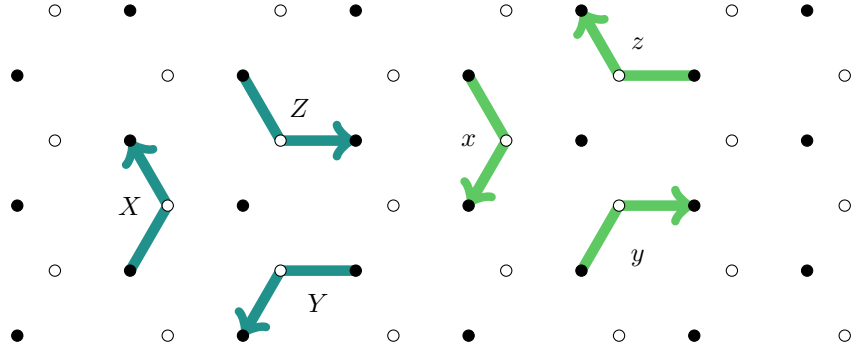

\iffalse
\begin{definition}
	Let $\tile{\partial R}$ be the \emph{tile boundary} (similar to the combinatorial boundary definition in \cite{conway-lagarias}). This is all cyclic permutations of $\partial R(e)$ for any edge $e$, as well as rotations by $120^\circ$ and all of these elements' inverses. 
\end{definition}

\begin{example}
		For the relation $YxZ$, then 
		\begin{align*}\tile{YxZ} = \{&YxZ, xZY, ZYx, & \text{cyclic permutations}\\
			&ZyX, yXZ, XZy, & \text{rotation by $120^\circ$}\\
			&XzY, zYX, YXz, & \text{rotation by $240^\circ$}\\
			&zXy, Xyz, yzX, & \text{inverses of line 1}\\
			&xYz, Yzx, zxY, & \text{inverses of line 3} \\
			&yZx, Zxy, xyZ \} & \text{inverses of line 2}
		\end{align*}
\end{example}
\fi
 
\begin{figure}[h!]\centering\small
	$\begin{array}{c c c}
		\begin{tikzpicture}[scale = 0.15]
			\def \hex {+(0, 0) -- +(2, 0) -- +(3, 1*1.73) -- +(2, 2*1.73) -- +(0, 2*1.73) -- +(-1, 1*1.73) -- +(0, 0)};
			\draw[very thin] (0, 0)     \hex;
			\draw[very thin] (-3, 1*1.73) \hex;
			\draw[very thin] (-6, 2*1.73) \hex;
			\draw[color = blue, line width = 2pt] (0, 0) -- ++(2, 0) -- ++(1, 1.73) -- ++(-1, 1.73) -- ++(-2, 0) -- ++(-1, 1.73) -- ++(-2, 0) -- ++(-1, 1.73) -- ++(-2, 0) -- ++(-1, -1.73) -- ++(1, -1.73) -- ++(2, 0) -- ++(1, -1.73) -- ++(2, 0) -- ++(1, -1.73) -- cycle;
			\node at (2, 0) {$\bullet$};
		\end{tikzpicture}
		&
		\begin{tikzpicture}[scale = 0.15]
			\def \hex {+(0, 0) -- +(2, 0) -- +(3, 1*1.73) -- +(2, 2*1.73) -- +(0, 2*1.73) -- +(-1, 1*1.73) -- +(0, 0)};
			\draw[very thin] (0, 0)     \hex;
			\draw[very thin] (0, 2*1.73) \hex;
			\draw[very thin] (0, 4*1.73) \hex;
			\draw[color = blue, line width = 2pt] (0, 0) -- ++(2, 0) -- ++(1, 1.73) -- ++(-1, 1.73) -- ++(1, 1.73) -- ++(-1, 1.73) -- ++(1, 1.73) -- ++(-1, 1.73) -- ++(-2, 0) -- ++(-1, -1.73) -- ++(1, -1.73) -- ++(-1, -1.73) -- ++(1, -1.73) -- ++(-1, -1.73) -- cycle;
			\node at (2, 0) {$\bullet$};
		\end{tikzpicture}
		&	
		\begin{tikzpicture}[scale = 0.15]
			\def \hex {+(0, 0) -- +(2, 0) -- +(3, 1*1.73) -- +(2, 2*1.73) -- +(0, 2*1.73) -- +(-1, 1*1.73) -- +(0, 0)};
			\draw[very thin] (-2, 0)     \hex;
			\draw[very thin] (1, 1*1.73) \hex;
			\draw[very thin] (4, 2*1.73) \hex;
			\draw[color = blue, line width = 2pt] (0, 0) -- ++(-2, 0) -- ++(-1, 1.73) -- ++(1, 1.73) -- ++(2, 0) -- ++(1, 1.73) -- ++(2, 0) -- ++(1, 1.73) -- ++(2, 0) -- ++(1, -1.73) -- ++(-1, -1.73) -- ++(-2, 0) -- ++(-1, -1.73) -- ++(-2, 0) -- ++(-1, -1.73) -- cycle;
			\node at (0, 0) {$\bullet$};
		\end{tikzpicture}
		\\
		ZZZYzzX%(\gamma^{-1}\beta^{-1})^3\alpha^{-1}(\gamma\beta)^3 \alpha 
		&
		ZxxYXXX%\gamma^{-1}(\beta^{-1}\alpha^{-1})^3\gamma(\beta\alpha)^3 
		& 
		ZYYYXyy%\beta^{-1}(\alpha^{-1}\gamma)^3 \beta (\alpha\gamma^{-1})^3
	\end{array}$
	
	$\begin{array}{c c}
		\begin{tikzpicture}[scale = 0.15]
			\def \hex {+(0, 0) -- +(2, 0) -- +(3, 1*1.73) -- +(2, 2*1.73) -- +(0, 2*1.73) -- +(-1, 1*1.73) -- +(0, 0)};
			\draw[very thin] (0, 0)     \hex;
			\draw[very thin] (3, -1*1.73) \hex;
			\draw[very thin] (3, 1*1.73) \hex;
			\draw[color = blue, line width = 2pt] (0, 0) -- ++(2, 0) -- ++(1, -1.73) -- ++(2, 0) -- ++(1, 1.73) -- ++(-1, 1.73) -- ++(1, 1.73) -- ++(-1, 1.73) -- ++(-2, 0) -- ++(-1, -1.73) -- ++(-2, 0) -- ++(-1, -1.73) -- ++(1, -1.73) -- cycle;
			\node at (5, -1.73) {$\bullet$};
		\end{tikzpicture}
		&
		\begin{tikzpicture}[scale = 0.15]
			\def \hex {+(0, 0) -- +(2, 0) -- +(3, 1*1.73) -- +(2, 2*1.73) -- +(0, 2*1.73) -- +(-1, 1*1.73) -- +(0, 0)};
			\draw[very thin] (-2, 0)     \hex;
			\draw[very thin] (-5, -1*1.73) \hex;
			\draw[very thin] (-5, 1*1.73) \hex;
			\draw[color = blue, line width = 2pt] (0, 0) -- ++(-2, 0) -- ++(-1, -1.73) -- ++(-2, 0) -- ++(-1, 1.73) -- ++(1, 1.73) -- ++(-1, 1.73) -- ++(1, 1.73) -- ++(2, 0) -- ++(1, -1.73) -- ++(2, 0) -- ++(1, -1.73) -- ++(-1, -1.73) -- cycle;
			\node at (-3, -1.73) {$\bullet$};
		\end{tikzpicture}
		\\
		ZZYYXX%(\gamma^{-1}\beta^{-1})^2(\alpha^{-1}\gamma)^2(\beta\alpha)^2 
		&
		yZxYzX%(\beta^{-1}\alpha^{-1})^2(\gamma\beta)^2(\alpha\gamma^{-1})^2
	\end{array}$
	
	$\begin{array}{c c c}
		\begin{tikzpicture}[scale = 0.15]
			\def \hex {+(0, 0) -- +(2, 0) -- +(3, 1*1.73) -- +(2, 2*1.73) -- +(0, 2*1.73) -- +(-1, 1*1.73) -- +(0, 0)};
			\draw[very thin] (0, -2*1.73) \hex;
			\draw[very thin] (3, -3*1.73) \hex;
			\draw[very thin] (6, -2*1.73) \hex;
			\draw[very thin] (9, -3*1.73) \hex;
			\draw[color = blue, line width = 2pt] (0, 0) -- ++(2, 0) -- ++(1, -1.73) -- ++(2, 0) -- ++(1, 1.73) -- ++(2, 0) -- ++(1, -1.73) -- ++(2, 0) -- ++(1, -1.73)  -- ++(-1, -1.73) -- ++(-2, 0) -- ++(-1, 1.73) -- ++(-2, 0) -- ++(-1, -1.73) -- ++(-2, 0) -- ++(-1, 1.73) -- ++(-2, 0) -- ++(-1, 1.73) -- cycle;
			\node at (11, -3*1.73) {$\bullet$};
		\end{tikzpicture}
		&
		\begin{tikzpicture}[scale = 0.15]
			\def \hex {+(0, 0) -- +(2, 0) -- +(3, 1*1.73) -- +(2, 2*1.73) -- +(0, 2*1.73) -- +(-1, 1*1.73) -- +(0, 0)};
			\draw[very thin] (-2, 0) \hex;
			\draw[very thin] (-2, 2*1.73) \hex;
			\draw[very thin] (-5, 3*1.73) \hex;
			\draw[very thin] (-5, 5*1.73) \hex;
			\draw[color = blue, line width = 2pt] (0, 0) -- ++(-2, 0) -- ++(-1, 1.73) -- ++(1, 1.73) -- ++(-1, 1.73) -- ++(-2, 0) -- ++(-1, 1.73) -- ++(1, 1.73) -- ++ (-1, 1.73) -- ++(1, 1.73) -- ++(2, 0) -- ++(1, -1.73) -- ++(-1, -1.73) -- ++(1, -1.73) -- ++(2, 0) -- ++(1, -1.73) -- ++(-1, -1.73) -- ++(1, -1.73) -- cycle;
			\node at (0, 0) {$\bullet$};
		\end{tikzpicture}
		&
		\begin{tikzpicture}[scale = 0.15]
			\def \hex {+(0, 0) -- +(2, 0) -- +(3, 1*1.73) -- +(2, 2*1.73) -- +(0, 2*1.73) -- +(-1, 1*1.73) -- +(0, 0)};
			\draw[very thin] (0, 0) \hex;
			\draw[very thin] (3, 1*1.73) \hex;
			\draw[very thin] (6, 0*1.73) \hex;
			\draw[very thin] (9, 1*1.73) \hex;
			\draw[color = blue, line width = 2pt] (0, 0) -- ++(2, 0) -- ++(1, 1.73) -- ++(2, 0) -- ++(1, -1.73) -- ++(2, 0) -- ++(1, 1.73) -- ++(2, 0) -- ++(1, 1.73)  -- ++(-1, 1.73) -- ++(-2, 0) -- ++(-1, -1.73) -- ++(-2, 0) -- ++(-1, 1.73) -- ++(-2, 0) -- ++(-1, -1.73) -- ++(-2, 0) -- ++(-1, -1.73) -- cycle;
			\node at (11, 1.73) {$\bullet$};
		\end{tikzpicture}
		\\
		ZyZZYzYzX%(\alpha\gamma^{-1}\beta^{-1}\gamma^{-1})^2\beta^{-1}(\alpha^{-1}\gamma\beta\gamma)^2\beta\ \  
		&
		ZxZxYXzXX%(\gamma^{-1}\beta^{-1}\alpha^{-1}\beta^{-1})^2\alpha^{-1}(\gamma\beta\alpha\beta)^2\alpha\ 
		&
		yZyZYYzYX%(\gamma^{-1}\alpha\gamma^{-1}\beta^{-1})^2\alpha^{-1}(\gamma\alpha^{-1}\gamma\beta)^2\alpha 
		\\
		\begin{tikzpicture}[scale = 0.15]
			\def \hex {+(0, 0) -- +(2, 0) -- +(3, 1*1.73) -- +(2, 2*1.73) -- +(0, 2*1.73) -- +(-1, 1*1.73) -- +(0, 0)};
			\draw[very thin] (0, 0) \hex;
			\draw[very thin] (-3, 1*1.73) \hex;
			\draw[very thin] (-3, 3*1.73) \hex;
			\draw[very thin] (-6, 4*1.73) \hex;
			\draw[color = blue, line width = 2pt] (0, 0) -- ++(2, 0) -- ++(1, 1.73) -- ++(-1, 1.73) -- ++(-2, 0) -- ++(-1, 1.73) -- ++(1, 1.73) -- ++(-1, 1.73) -- ++(-2, 0) -- ++(-1, 1.73) -- ++(-2, 0) -- ++(-1, -1.73) -- ++(1, -1.73) -- ++(2, 0) -- ++(1, -1.73) -- ++(-1, -1.73) -- ++(1, -1.73) -- ++(2, 0) -- cycle;
			\node at (2, 0) {$\bullet$};
		\end{tikzpicture}
		&
		\begin{tikzpicture}[scale = 0.15]
			\def \hex {+(0, 0) -- +(2, 0) -- +(3, 1*1.73) -- +(2, 2*1.73) -- +(0, 2*1.73) -- +(-1, 1*1.73) -- +(0, 0)};
		\draw[very thin] (0, 0) \hex;
		\draw[very thin] (0, 2*1.73) \hex;
		\draw[very thin] (3, 3*1.73) \hex;
		\draw[very thin] (3, 5*1.73) \hex;
		\draw[color = blue, line width = 2pt] (0, 0) -- ++(2, 0) -- ++(1, 1.73) -- ++(-1, 1.73) -- ++(1, 1.73) -- ++(2, 0) -- ++(1, 1.73) -- ++(-1, 1.73) -- ++ (1, 1.73) -- ++(-1, 1.73) -- ++(-2, 0) -- ++(-1, -1.73) -- ++(1, -1.73) -- ++(-1, -1.73) -- ++(-2, 0) -- ++(-1, -1.73) -- ++(1, -1.73) -- ++(-1, -1.73) -- cycle;
		\node at (2, 0) {$\bullet$};
		\end{tikzpicture}
		&
		\begin{tikzpicture}[scale = 0.15]
			\def \hex {+(0, 0) -- +(2, 0) -- +(3, 1*1.73) -- +(2, 2*1.73) -- +(0, 2*1.73) -- +(-1, 1*1.73) -- +(0, 0)};
			\draw[very thin] (-2, 0) \hex;
			\draw[very thin] (1, 1*1.73) \hex;
			\draw[very thin] (1, 3*1.73) \hex;
			\draw[very thin] (4, 4*1.73) \hex;
			\draw[color = blue, line width = 2pt] (0, 0) -- ++(-2, 0) -- ++(-1, 1.73) -- ++(1, 1.73) -- ++(2, 0) -- ++(1, 1.73) -- ++(-1, 1.73) -- ++(1, 1.73) -- ++(2, 0) -- ++(1, 1.73) -- ++(2, 0) -- ++(1, -1.73) -- ++(-1, -1.73) -- ++(-2, 0) -- ++(-1, -1.73) -- ++(1, -1.73) -- ++(-1, -1.73) -- ++(-2, 0) -- cycle;
			\node at (3, 1.73) {$\bullet$};
		\end{tikzpicture}
		\\
		ZZxZYzXzX%\gamma^{-1}(\beta^{-1}\gamma^{-1}\beta^{-1}\alpha^{-1})^2\gamma(\beta\gamma\beta\alpha)^2 
		&
		ZxYxYXXyX%\beta^{-1}(\alpha^{-1}\beta^{-1}\alpha^{-1}\gamma)^2\beta(\alpha\beta\alpha\gamma^{-1})^2 
		& 
		yZYxYYXyX%\gamma^{-1}(\beta^{-1}\alpha^{-1}\gamma\alpha^{-1})^2\gamma(\beta\alpha\gamma^{-1}\alpha)^2
		
	\end{array}$\normalsize
	\caption{Figure \ref{alltiles} after rewriting in terms of $X$, $Y$, and $Z$}
	\label{alltilesXYZ}
\end{figure}
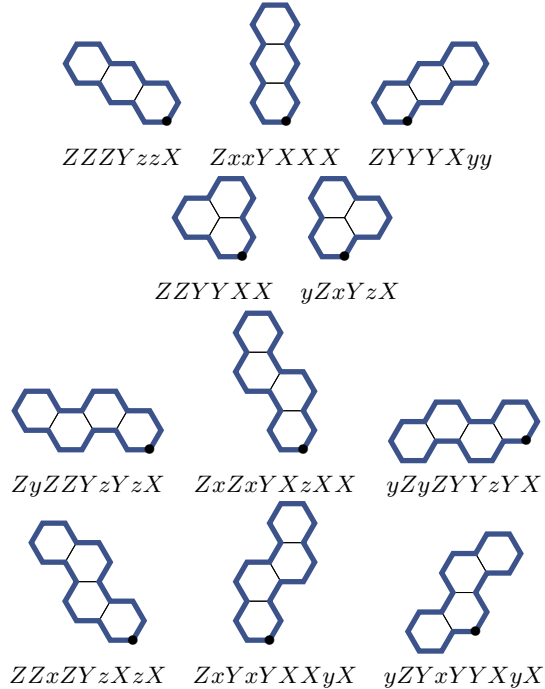

At this point, we can calculate that the tile group $T(\Sigma) = \{I\}$. This leads us to wonder whether $G$ might be $h(\Sigma)$, the tile homotopy group from Definition \ref{homotopy}. At a glance, this appears to be true. To check, we use the computer to check some (a few thousand) short examples. 

We see rapidly that there are some relations, such as $X^6$, whose product is $I$, and yet do not form a closed loop. We also find some relations, like such as $XYZ$, that do in fact return to their starting vertex and contribute $I$, but  have zero area (in terms of cells of the hexagonal grid). Rather than try to count paths by area, or figure out which loops correspond to regions of positive area, we take a different approach.

\begin{definition}
	Let $\tile{\partial R}$ be the \textbf{tile boundary} (similar to Definition \ref{conwaydef}). This is all cyclic permutations of $\partial R(e)$ for any edge $e$, as well as rotations by $120^\circ$ and all of these elements' inverses. 
\end{definition}

\begin{example}
	For the relation $YxZ$, then 
	\begin{align*}\tile{YxZ} = \{&YxZ, xZY, ZYx, & \text{cyclic permutations}\\
		&ZyX, yXZ, XZy, & \text{rotation by $120^\circ$}\\
		&XzY, zYX, YXz, & \text{rotation by $240^\circ$}\\
		&zXy, Xyz, yzX, & \text{inverses of line 1}\\
		&xYz, Yzx, zxY, & \text{inverses of line 3} \\
		&yZx, Zxy, xyZ \} & \text{inverses of line 2}
	\end{align*}
\end{example}

\subsection*{Methodology}

We first create a list of all words of length up to 15 in the elements $\{X, Y, Z, x, y,z\}$. Since the hexagonal grid has nice rotation symmetries, we can consider paths that start with $X$, so we append $X$ to the end of the generated words. This means we have examined all paths through edge length 32 that start and end at shaded vertices. Then we calculate the contribution of each possible path to see if it contributes $\pm I$, the positive or negative identity matrix. If it does, it is saved in a file for further processing. 

Once we have a list of all identity relations up to a given length, we start with the shortest length terms and append them to a reduction list. For example, since $YZX = 1$, then $x = YZ$, so any string containing $YZ$ could be shortened by replacing $YZ$ with $x$. So any string containing $YZ$ is eliminated from our minimal relation list. We continue this process of removing relations that could be rewritten and shortened by previous strings. We arrive at the list given in Table \ref{firstlist}. We have labeled the stone, bone, and snake tiles for convenience. 
Since each of these strings is related to a path on the hexagon lattice, we can also visualize them as in Figure \ref{ending points}. 
\begin{figure}[h!]
	\centering
	\includegraphics[width = .35\textwidth]{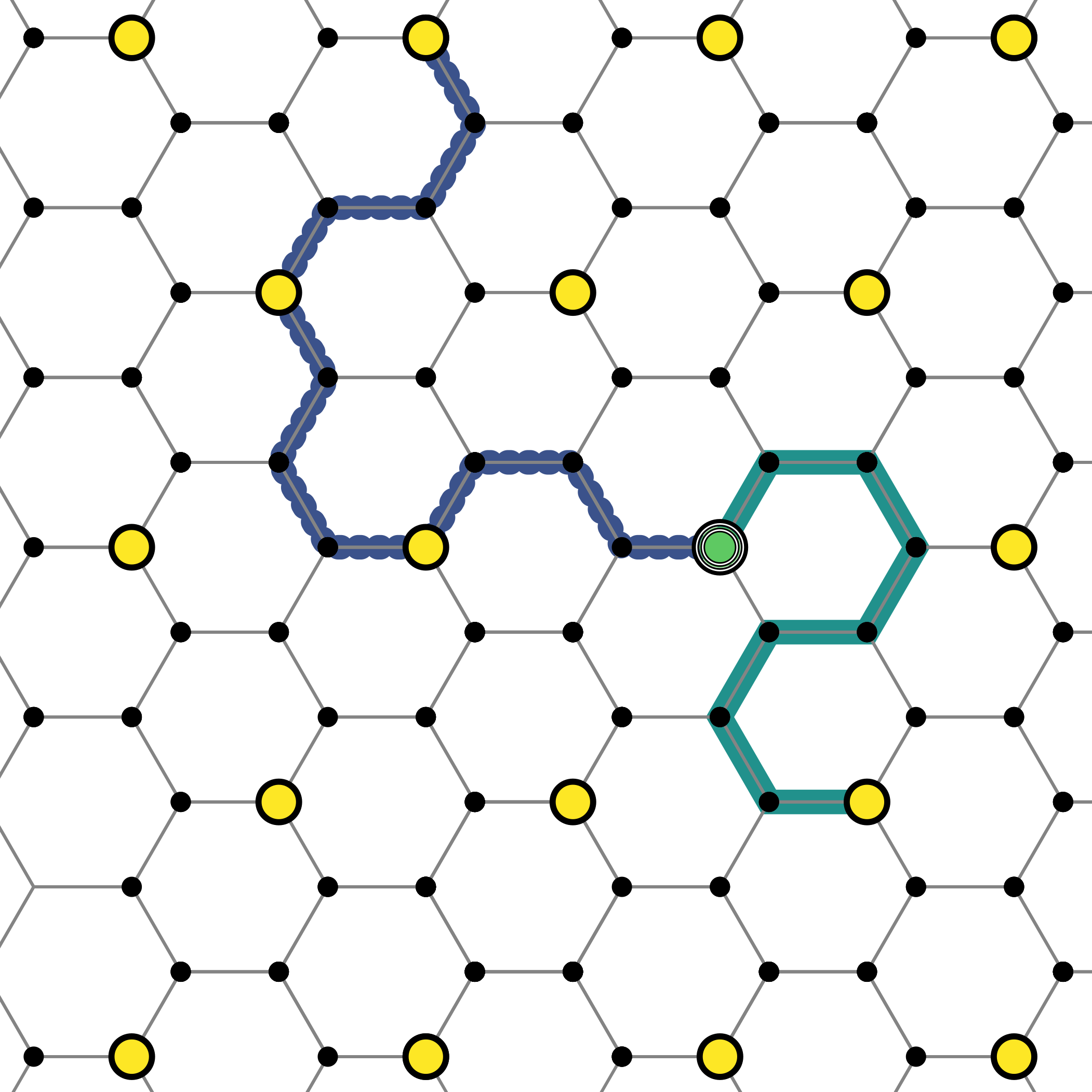}
	\caption[Start and end points of identity paths]{A path $P$ beginning at the green multiply-outlined vertex will only have that $P = I$ if it ends at one of the yellow vertices. An example of such a path is shown in teal with smooth edges, and another example in dark blue with bumpy edges.}
	\label{ending points}
	
\end{figure}

\begin{table}[h!]
	\centering
	$\begin{array}{c c}
		\tile{YZX}\\
		\tile{ZyzX}\\
		\tile{XZZyX}\\
		\tile{yZxYzX} & STONE\\
		\tile{yZZyyX}\\
		\tile{zzYYzX}\\
		\tile{ZyZYzYX} & (barbell)\\
		\tile{ZxxYXXX} & BONE\\
		\tile{yXXyXXyXX}\\
		\tile{XzXXzXXzX}\\
		\tile{XyZZxYYzX} & (2 by 2 by 2)\\
		\tile{yZyZYYzYX} & SNAKE \\
		\tile{ZZxZYzXzX} & SNAKE\\
	\end{array}$
	
	\caption[List of relations]{The candidate list of relations generated by computer experimentation, with tiles labeled}
	\label{firstlist}
\end{table}

Now, there are two important closed loops that still arise in this data, namely $ZyZYzYX$, which we call the barbell, and $zzYxxZyyX$, the 2 by 2 by 2 hexagon graph. This brings us to the following two lemmas, which have been a great help in terms of reducing the number of required relations.

	\begin{lemma}
		The hexagon barbell (as in Figure \ref{barbelltiling}) is signed-tilable by only bones and snakes, and contributes the identity matrix to a path product.
	\end{lemma}
	
	\begin{proof}
		See Figure \ref{barbelltiling} for an example of a signed tiling of the barbell shape. Therefore, by Corollary \ref{tilingtotrace}, $\partial R = I$. 
		\begin{figure}[h!]
			\centering
			\includegraphics[width = 0.75\textwidth]{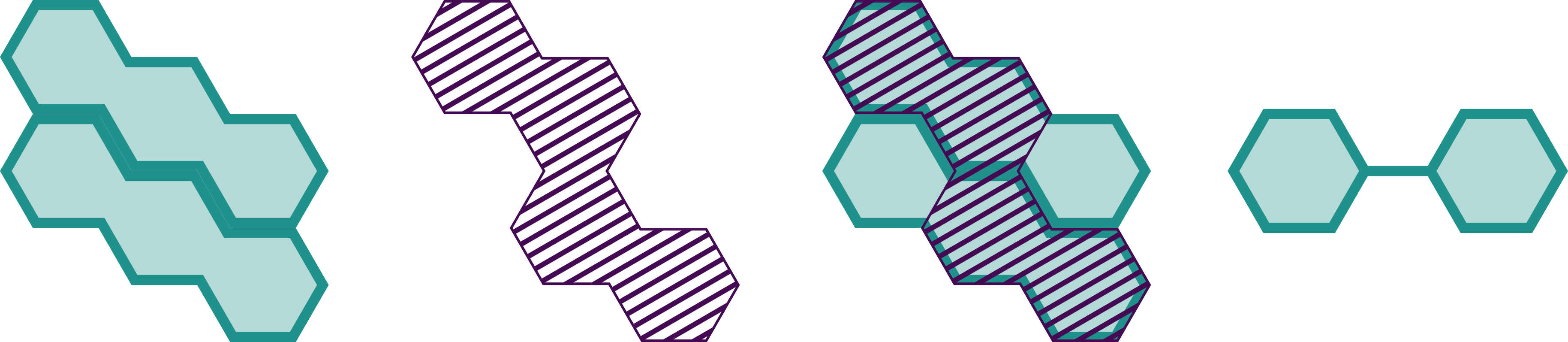}
			\caption{A signed tiling of the `barbell' shape}
			\label{barbelltiling}
		\end{figure}
	\end{proof}

	\begin{lemma}
		The loop around a 2 by 2 by 2 region on a hexagonal grid (as in Figure \ref{2x2x2}) is signed-tilable by only bones and snakes, and contributes the identity matrix to a path product.
	\end{lemma}
	
	\begin{proof}
		See Figure \ref{2x2x2} for an example of a signed tiling of the 2 by 2 by 2 shape. Therefore, by Corollary \ref{tilingtotrace}, $\partial R = I$. 
	\end{proof}

We are now left with a remaining seven words that do not correspond to a closed loop. We check each of these words by hand, seeing how and whether they can be reduced or rewritten. The results of this calculation is shown in Table \ref{reduction}.

\begin{table}[h!]
	\footnotesize
	\begin{align*}
		YZX &= -YYY & \text{ since $ZX = -YY$ from the bone relation}\\
		\\
		ZyzX &= yzXZ & \text{ cyclic permutation} \\
		&= -yzXzz & \text{ since $Z = -zz$ from $XXX = -1$}\\
		& = -xZxZ & \text{ from the snake relation $yzXzzxZxZ = 1$}\\
		\\
		XZZyX &= X(-XY)yX & \text{ since $ZZ = -XY$ from bone relation} \\
		&= -XXX \\
		\\
		yZZyyX &= y(-XY)yyX & \text{ since $ZZ = -XY$ from bone relation} \\
		&= -yXyX \\
		\\
		zzYYzX &= zz(-ZX)zX & \text{ since $YY = -ZX$ from bone relation} \\
		&= -zXzX \\
		\\
		yXXyXXyXX &= y(-YZ)y(-YZ)y(-YZ) & \text{ since $XX = -YZ$ from bone relation} \\
		& = -yyy \\
		\\
		XzXXzXXzX &= z(-YZ)z(-YZ)z(-YZ) & \text{ since $XX = -YZ$ from bone relation} \\
		&= -zYYYZ \\
		& = -YYY & \text{ cyclic permutation (and $Zz = 1$)}
	\end{align*}
	\caption{Reduction of each open string from Table \ref{firstlist}}
	\label{reduction}
\end{table}

So each of the non-closed paths ends at either $\tile{XXX}$ or $\tile{zXzX}$. These are two very important strings which come from `half' of the bone and snake relations, respectively. $\tile{XXX} = -1 = \tile{zXzX}$. We can then verify that the other `half' of each bone and snake tile is also $-1$ in the matrix product. So our group can be written in one of the two ways given in Conjecture \ref{finitegroup}. 

\begin{conjecture}
	Let $I$ be the $2 \times 2$ identity matrix, and $-I$ be the $2 \times 2$ matrix with $-1$ on the diagonal. Then $G$ has either of the following two finite presentations:
	$${G} = \lr{X, Y, Z \vert \tile{XXX} = \tile{yXyX} = \tile{yzYX} = \tile{YzzX} = \tile{yZxYzX} = -I}$$ ~or, equivalently~ $${G} = \lr{X, Y, Z \left\vert \begin{matrix}
			\tile{ZxxYXXX} = \tile{yZyZYYzYX} = \tile{ZZxZYzXzX} = I, \\
			\tile{XXX} = \tile{yZxYzX} = \tile{yXyX} = -I\end{matrix}\right.}$$ 
		\label{finitegroup}
\end{conjecture}

\printbibliography

\end{document}